\documentclass[11pt]{article}

\usepackage{latexsym,mathrsfs}
\usepackage{amsmath,amssymb} 
\usepackage{amsthm,enumerate,verbatim}
\usepackage{amsfonts}
\usepackage{graphicx}
\usepackage{algorithm}
\usepackage{algorithmic}
\usepackage{url}
\usepackage{color}
\usepackage{hyperref}
\usepackage{authblk}

\newcommand {\mat}  [1] {\left[\begin{array}{#1}}
	\newcommand {\rix}      {\end{array}\right]}

\newtheorem{lemma}{Lemma}

\newtheorem{theorem}{Theorem}

\newtheorem{definition}{Definition}
\newtheorem*{definition*}{Definition}
\newtheorem{example}{Example}

\newtheorem{remark}{Remark}
\newtheorem*{problem*}{Problem}

\DeclareMathOperator{\Sym}{Sym} 
\DeclareMathOperator{\Skew}{Skew}

\def \R{{\mathbb R}}
\def \C{{\mathbb C}}

\definecolor{brightpink}{rgb}{1.0, 0.0, 0.5}

\title{Computing the nearest $\Omega$-admissible descriptor dissipative Hamiltonian system} 
\date{}

\author{
	Vaishali Aggarwal\thanks{Indian Institute of Technology Delhi, Hauz Khas, 110016 New Delhi, India. Email: maz238696@maths.iitd.ac.in.} 
	\qquad 
	Nicolas Gillis\thanks{University of Mons, Rue de Houdain 9, 7000 Mons, Belgium. Email: nicolas.gillis@umons.ac.be. NG acknowledges the support  by the European Union (ERC consolidator, eLinoR, no 101085607).} \qquad 
	Punit Sharma\thanks{Indian Institute of Technology Delhi, Hauz Khas, 110016 New Delhi, India. Email: punit.sharma@maths.iitd.ac.in. PS acknowledges the support of the SERB- CRG grant (CRG/2023/003221) and SERB-MATRICS grant by Government of India.}
}

\begin{document}
			
			\maketitle
			
\begin{abstract}
For a given set $\Omega \subseteq \mathbb{C}$, a matrix pair $(E,A)$ is called $\Omega$-admissible if it is regular, impulse-free and its eigenvalues lie inside the region $\Omega$. In this paper, we provide a dissipative Hamiltonian characterization for the matrix pairs that are $\Omega$-admissible where $\Omega$ is an LMI region. We then use these results for solving the nearest $\Omega$-admissible matrix pair problem:  Given a matrix pair $(E,A)$, find the nearest $\Omega$-admissible pair $(\tilde E, \tilde A)$ to the given pair $(E,A)$. We illustrate our results on several data sets and compare with the state of the art. 

The code is available from \url{https://gitlab.com/ngillis/nearest-omega-stable-pair}. 
\end{abstract}

			
\section{Introduction}
In this paper, we study the $\Omega$-admissibility of linear time-invariant \emph{descriptor systems} of the form
\begin{equation}\label{eqa0}
	E\dot{x}(t) = Ax(t) + f(t)
\end{equation}
on the unbounded interval $t \in I=[t_0,\infty)$, where $E,A \in \R^{n,n}$,  and $f$ is a sufficiently smooth function from $I$ to $\R^n$. Systems of the form~\eqref{eqa0} arise from linearization around stationary solutions of the initial value problems (IVPs) for general implicit systems of differential-algebraic equations~\cite{KunM06}. We use the matrix pair $(E,A)$ to represent the descriptor system~\eqref{eqa0}. The system~\eqref{eqa0} is called a \emph{standard system} if $E=I_n$, where $I_n$ is the identity matrix of size $n$. Descriptor systems are known for their complex structures because they contain both finite and infinite poles which may generate undesired impulsive behaviours. Thus, in the study of such systems, regularity and the absence of impulses need to be guaranteed~\cite{KunM06}. 

The system~\eqref{eqa0} is said to be \emph{regular} if the matrix pair $(E,A)$ is regular, that is, $\text{det}(sE-A) \neq 0$ for some $s \in \C$, otherwise it is called \emph{singular}. Regularity ensures that the IVP of solving~\eqref{eqa0} with a consistent initial value $x_0$ has a unique solution.
 For a regular system~\eqref{eqa0}, the roots of the polynomial $\text{det}(sE-A) $ are called the \emph{eigenvalues} of the matrix pair $(E,A)$. A regular matrix pair $(E,A)$ has $\infty$ as an eigenvalue if $E$ is singular. The system~\eqref{eqa0} is \emph{impulse-free} (that is, of index at most one) if $(E,A)$ has $\text{rank}(E)$ number of finite eigenvalues. Impulse-freeness ensures that the impulse response of system is bounded over time. The stability of the system~\eqref{eqa0} is determined by the location of the finite eigenvalues of the matrix pair $(E,A)$. 
A matrix pair $(E,A)$ is said to be \emph{$\Omega$-stable} if it is regular and all its finite eigenvalues lie inside a given region $\Omega$ in the complex plane. 
The admissibility concept of system~\eqref{eqa0} can be investigated by choosing different regions inside which those finite eigenvalues will lie~\cite{Var95,KunM06}.
\begin{definition}[$\Omega$-admissibility]\label{def1a}
	For  $\Omega \subseteq \mathbb{C}$, the matrix pair $(E,A)$, where $E, A\in\mathbb{R}^{n,n}$ is said to be \emph{$\Omega$-admissible} if the pair is regular, impulse-free and the finite eigenvalues of the matrix pair $(E,A)$ lie in the region $\Omega$.
\end{definition}

The stability and performance of the control systems
are closely related to its $\Omega$-admissibility, 
as it is directly linked to the placement of closed-loop poles in a suitable region in the complex plane. It helps in optimizing the system in terms of speed, accuracy, and robustness, and provides flexibility in the system design~\cite{r4,r9,r10,r14,r5,r6}. The step response of a system with eigenvalue $\lambda=-\tau w_n \pm i w_d$ is fully categorized in terms of the undamped natural frequency $w_n=|\lambda|$, the damping ratio $\tau$ and the damped natural frequency $w_d$. By confining $\lambda$ to lie in a prescribed region, specific bounds can be put on these quantities to ensure a satisfactory transient response, see~\cite{TehWE15,KawS83,WanF06} and references therein for standard systems. For example, by restricting the eigenvalues in the intersection of a shifted half plane, a sector and a disk, system limits the maximum overshoot, the frequency of oscillatory modes, the delay time, the rise time, and the settling time~\cite{r10}.

Our main focus in this paper is on the regions of the complex plane which can be expressed using linear matrix inequalities (LMIs). The LMI regions cover a large variety of useful regions, including half-planes, disks, sectors, vertical/horizontal strips, ellipsoid, parabolic regions, hyperbolic regions,  and any intersection thereof. 
This class of regions can be successfully applied in solving some robust control problems based on LMI structures of regions.  LMI approaches are suitable for applications because there are effective  algorithms, such as interior-point methods, for the
solution of LMI problems~\cite{r16}. 
The conditions for $\Omega$-admissibility can often be expressed as LMIs, which are computationally tractable. Due to the importance of the problem, LMI-based (robust) pole clustering characterization for descriptor systems has attracted much
attention in the past years~\cite{r4,r14,r5,r6,r19,r20}. 

The aim of this paper is two fold: (i)~provide a dissipative Hamiltonian (DH) characterization for the set of matrix pairs which are regular, impulse-free, and whose eigenvalues  belong to a given LMI region. This is a generalization of the work in~\cite{ChouGS24}, where a DH characterization  was obtained for matrices with eigenvalues in an LMI region, and (ii)~find the nearest  $\Omega$-admissible matrix pair to a given matrix pair. More precisely, solve the following optimization problem:
\begin{problem*}
	For a given region $\Omega \subseteq \C$ and a matrix pair $(E,A)$, where $E,A \in \R^{n,n}$, find the nearest $\Omega$-admissible matrix pair $(\tilde E, \tilde A)$, that is, solve
\begin{equation}\label{eqa1}
	\inf _{(\tilde{E},\tilde{A}) \in \mathcal{S}_{\Omega}^a}{\|E-\tilde{E}\|}_{F}^{2}+ {\|A-\tilde{A}\|}_{F}^{2}, \tag{$\mathcal P$}
\end{equation}
where $\mathcal{S}_{\Omega}^a$ is the set of all $\Omega$-admissible matrix pairs of size $n \times n$ and ${\|\cdot\|}_{F}$ stands for the Frobenius norm.
\end{problem*}

This problem is closely related to the problem of finding the nearest $\Omega$-stable pairs, which was recently studied in~\cite{noferini2025finding}. In $\Omega$-stability, one aims to find the nearest matrix pair that is regular and has all its eigenvalues inside a given $\Omega$-region. However, it does not guarantee that the computed solution is of index-1 or impulse-free. While in the $\Omega$-admissibility problem~\eqref{eqa1}, the computed solution must be $\Omega$-stable and impulse-free.

In~\cite{noferini2025finding}, authors studied the $\Omega$-stability problem for complex pairs, that is, $(E,A)\in (\C^{n,n})^2$, where they proposed a method based on Riemannian optimization that uses the Schur form of complex pairs. Further, their method works, in theory, for any closed $\Omega$-region in the complex plane; however, in practice, it requires computing a projection on the feasible set that becomes extremely technical and challenging if $\Omega$ is not the unit disk or the left half plane. Our work differs from~\cite{noferini2025finding} in the following ways: 
\begin{itemize}
    \item The problem considered here is for real pairs $(E,A)$.
    \item It provides a DH-parametrization of $\Omega$-admissible pairs, where $\Omega$ is an LMI region.
    \item The $\Omega$-admissibility ensures $\Omega$-stability and impulse-freeness.
    \item DH parametrization leads to a computable projection of matrix pairs on LMI regions.
    \item In some cases, our method provides better results even for the $\Omega$-stability problem than the Riemannian optimization from~\cite{noferini2025finding}.
\end{itemize} 

The $\Omega$-admissibility problem~\eqref{eqa1} is useful in system identification, where a given system must be modified to satisfy performance constraints while keeping it as close as possible to the original system, see~\cite{r1,GilKS20,ChouGS24,r3} for some other related nearness problems. 

This paper is organized as follows: 
Section~\ref{sec:prelim} provides some preliminaries on the admissibility of the descriptor system and on LMI regions. We also define a DH matrix pair, which has the form $(E,A) = (E,(J-R)Q)$ where $J^\top = -J$, $R = R^\top$ and $Q$ is invertible. 
In Section~\ref{sec:DHchara}, we consider LMI regions in the left half of the complex plane. We first give a sufficient condition for a DH pair to be $\Omega$-admissible. We then provide a parametrization for the set of all $\Omega$-admissible pairs using DH pairs when $\Omega$ is a uniform LMI region. The uniform LMI regions are nonempty LMI regions which are cones in the complex plane. 
In Section~\ref{sec:specialomega}, we allow LMI regions to intersect with the right half of the complex plane and provide a class of matrix pairs 
$(E,(J-R)Q)$ with symmetry and semidefiniteness constraints on quadruple  $(E,J,R,Q)$ that are regular and have eigenvalues inside these regions. 
In Section~\ref{sec:omegastable}, we show numerical experiments for Hurwitz and Schur stable pairs, as well as more general LMI regions, and compare our results with the state of the art. 


\paragraph{Notation} In this paper, $I_{n}$ denotes the identity matrix of size $n \times n$. For a symmetric matrix $M$, $M\succ 0$ (M $\prec 0$) and $M\succeq 0$ ($M\preceq 0$), respectively, stand for positive definite (negative definite) and positive semidefinite (negative semidefinite) matrix. We use $\|\cdot\|$ for the spectral norm of a matrix or a vector, and $\Re(\lambda)$ and $\Im(\lambda)$, respectively, denote the real and imaginary parts of a complex number $\lambda$. The complex conjugate transpose of a matrix or a vector $M$ is denoted by $M^{*}$. The  Kronecker product of two matrices $A$ and $B$ is denoted by $A \otimes B$. The standard properties of the Kronecker product are detailed in~\cite{Tis83}.

\section{Preliminaries}		\label{sec:prelim}
In this section, we state results that will be useful in obtaining a parameterization for $\Omega$-admissible matrix pairs with eigenvalues in a prescribed region $\Omega$ in the complex plane. For this, we first define a DH matrix pair.
\begin{definition}[DH matrix pair]\label{defdhpair} 
	A matrix pair $(E,A)$ with $E,A \in \R^{n,n}$ is called a \emph{DH matrix pair}, if $A=(J-R)Q$
for some $J,R,Q \in \R^{n,n}$ such that $J^\top =-J$, $R \succ 0$, and $Q$ invertible with $Q^\top E=E^\top Q \succeq 0$.	
\end{definition}
The matrix $R$ in a DH matrix pair $(E,(J-R)Q)$ is called the \emph{dissipation matrix}. We note that this definition of a DH matrix pair is slightly more restrictive than that of a DH matrix pair in~\cite{r1}, where it is not require that the matrix $R$ is positive definite. 

The following result from~\cite{r1} gives an equivalent condition for a matrix pair to be admissible. 

\begin{theorem}\label{thm:equiasym}
	Let $(E,A)$ be a matrix pair, where $E,A \in \R^{n,n}$. Then the following are equivalent.
	\begin{enumerate}
		\item $(E,A)$ is regular, impulse-free, and asymptotically stable.
		\item $(E,A)$ is a DH matrix pair.
		\item There exists an invertible matrix $P \in \R^{n,n}$ such that 
 \begin{equation*}\label{eq7}
	E^\top  P = P^\top  E \succeq 0 \text{ and } A^{\top}P+P^{\top}A\prec0.
\end{equation*}
	\end{enumerate}
\end{theorem}		

The following lemma characterizes real and imaginary parts of a finite eigenvalue of a DH matrix pair that will be used in the parametrization of the conic sector regions in Section~\ref{sec:DHchara}.

\begin{lemma}\label{lem2} 
	Let $(E, (J - R)Q) \in (\R^{n,n})^2$ be a regular matrix pair, where $J^\top =-J$, $R^\top =R$, and $Q$ invertible with $Q^\top E=E^\top Q \succeq 0$, and let $\lambda \in \mathbb{C}$ and $x \in \mathbb{C}^{n} \setminus \{0\}$ be such that $x^*(J-R)Q=\lambda x^*E$. Then
\[	
	\Re(\lambda)=-\frac{x^{*} R x}{x^{*} EQ^{-1} x} \quad \text { and } \quad \Im(\lambda)=-i \frac{x^{*} J x}{x^{*} EQ^{-1} x}.
	\]
\end{lemma}
\begin{proof}
	Let $x$ be a left eigenvector of $(E,(J-R)Q)$ corresponding to the eigenvalue $\lambda$, that is, $x$ satisfies \mbox{$x^{*} (J-R)Q=\lambda x^{*}E$}. Since $Q$ is invertible, 
	\begin{equation}\label{eq:lem1eq1}
		x^{*}(J-R)x=\lambda x^{*}E Q^{-1}x,
	\end{equation}
	and by taking the conjugate transpose of~\eqref{eq:lem1eq1}, we get 
		\begin{equation}\label{eq:lem1eq2}
		-x^{*}(J+R)x=\overline \lambda x^{*}E Q^{-1}x,
	\end{equation}
	where we used the fact that $E^\top Q \succeq 0$ implies that $EQ^{-1} \succeq 0$. In view of~\eqref{eq:lem1eq1} and~\eqref{eq:lem1eq2}, we have
	\begin{equation}\label{eq:lem1eq3}
		x^*Rx=-\Re(\lambda)x^*EQ^{-1}x \quad \text{and}\quad 	x^*Jx=i \Im(\lambda)x^*EQ^{-1}x.
	\end{equation}
	Note that $x^*EQ^{-1}x \neq 0$, because if $x^*EQ^{-1}x =0$, then $x^*EQ^{-1}=0$ as $EQ^{-1} \succeq 0$, which implies that $x^*(J-R)=0$. This implies that $(J-R)Q$ and $E$ has a common left null space, and thus $(E,(J-R)Q)$ becomes singular. This is a contradiction to the assumption that $(E,(J-R)Q)$ is regular. 
 Thus from~\eqref{eq:lem1eq3}, we have
\[	
	\Re(\lambda)=-\frac{x^{*} R x}{x^{*} EQ^{-1} x} \quad \text { and } \quad \Im(\lambda)=\frac{x^{*} J x}{i x^{*} EQ^{-1} x}= -i\frac{x^{*} J x}{ x^{*} EQ^{-1} x}.
\] 
\end{proof}

\subsection{LMI regions}\label{subsec:lmi}

In this section, we briefly discuss LMI regions and their properties.
A subset $\Omega \subseteq \C$ is called an LMI region if it can be expressed by linear matrix inequalities (LMIs). The LMI regions, considered in this paper are defined as follows:

\begin{definition*}
	A subset  $\Omega \subseteq \C$ is called an LMI region if there exist real matrices
	 $B \in \R^{s,s}$ and $C \in \R^{s,s}$ such that $B^\top =B$ and 
	\begin{equation}\label{eq:lmidef}
	\Omega=\left\{z \in \C:~f_\Omega(z) \prec 0\right\}, \quad 	\text{where}~ f_\Omega(z)=B+Cz+C^\top \overline z.
	\end{equation}
\end{definition*}
The \emph{characteristic function} $f_\Omega(z)$ can also be written as 
\[
f_\Omega(z)=B+\Sym(C)x+ \Skew(C)iy,
\]
where $x=\Re(z)$, $y =\Im(z)$, $\Sym(C)=C+C^\top$, and $\Skew(C)=C-C^\top$. The characteristic function of an LMI region is not unique~\cite{r10}. 

An LMI region $\Omega$ is called a uniform LMI region if $B=0$, that is, $f_\Omega(z)=Cz+C^\top  \overline{z}$. For example, the conic sector region $\Omega=\left\{z=x+iy \in \C:~x<0, -x \tan(\theta)<y<x \tan(\theta)\right\}$, where $0<\theta < \frac{\pi}{2}$ is a uniform LMI region with the characteristic function 
\[
f_\Omega(z)=\mat{cc}\sin (\theta) & \cos(\theta)\\ -\cos(\theta) & \sin(\theta)\rix z+
\mat{cc}\sin (\theta) & -\cos(\theta)\\ \cos(\theta) & \sin(\theta)\rix \overline{z}.
\]
In fact, any non-empty conic region symmetric to the real axis is a uniform LMI region. We will denote $\Omega_u = \left\{z \in \C:~Cz+C^\top \overline z \prec 0 \right\}$ the uniform part of $\Omega$ (note that $\Omega_u = \Omega$ when $B = 0$). 
LMI regions are convex and symmetric with respect to the real axis. The intersection of two LMI regions is again an LMI region; see Remark~\ref{rem:uniflminter} below. 
LMI regions are dense in the set of convex regions that are symmetric with respect to the real axis. We refer to~\cite{Kus19} and references therein for more detailed properties of LMI regions. 

A large number of regions which are relevant for control systems can be expressed as LMI regions, for example, conic sectors, vertical strips, discs, horizontal strips, ellipses, parabolic regions, hyperbolic sectors and their intersections~\cite{ChilG96,BisPT21}, see Section~\ref{sec:specialomega} for some specific LMI regions.

The $\Omega$-pole placement problem and related problems of matrix or matrix pair $\Omega$-stability with respect to a given LMI region $\Omega$ have appeared in many applications~\cite{r4,r14,r5,r6, r19, r20}. Recently, in~\cite{ChouGS24} and~\cite{noferini2025finding}, the $\Omega$-stable matrix problem has been studied via DH systems and Riemannian optimization, respectively. 

Motivated by~\cite{ChouGS24}, in this paper, we propose a DH parametrization for $\Omega$-admissible matrix pairs, where $\Omega$ is a given LMI region. The following result will be crucial for deriving a DH parametrization for matrix pairs which are regular, impulse free, and have eigenvalues in a uniform LMI region.

\begin{theorem}{\rm \cite{r4}}\label{thm:uniflmi}
	Let $\Omega \subseteq \C$ be an LMI region and $(E,A) \in (\R^{n,n})^2$. Then the following are equivalent. 
	\begin{enumerate}
    
		\item The set $\Omega_u = \left\{z \in \C:~Cz+C^\top \overline z \prec 0 \right\}$ is nonempty, and the matrix pair $(E,A)$ is $\Omega$-admissible. 
        
		\item There exists a matrix $X$ such that
		\begin{equation}\label{eq:unilmi1}
			E^\top X=X^\top E \succeq 0  \quad \text{and}\quad M_{\Omega}(E,A,X):=B \otimes E^\top X+C \otimes X^\top A + C^\top  \otimes A^\top X \prec 0.
		\end{equation}
        
	\end{enumerate}
\end{theorem}

\newpage 

\begin{remark}{\rm 
	We note that 
	\begin{itemize}
		\item  The matrix $X$ that satisfies~\eqref{eq:unilmi1} is always invertible. This follows from the fact that $x \in \C^{n}$ satisfying $Xx=0$ implies that $M_{\Omega}(E,A,X) (I \otimes x)=0$. 
		\item A uniform LMI region is nonempty if and only if $\Sym(C)$ is negative or positive definite~\cite{Kus19}. Thus, Theorem~\ref{thm:uniflmi} is useful for LMI regions with negative or positive definite $\Sym(C)$.
	\end{itemize}
	}
\end{remark}

\begin{remark}\label{rem:uniflminter}{\rm
		Let $\hat \Omega$ and $\tilde \Omega$ be two LMI regions with characteristic functions 
		$f_{\hat \Omega}(z)=\hat B+\hat Cz+\hat C^\top \overline z$ and $f_{\tilde \Omega}(z)=\tilde B+\tilde Cz+\tilde C^\top \overline z$, respectively. 
        Then, the intersection $\Omega= \hat \Omega \cap \tilde \Omega $ is also an LMI region with the characteristic function $f_{\Omega}(z)= B+ Cz+ C^\top \overline z$, where 
		\[
		B=\mat{cc} \hat B & 0\\0 & \tilde B \rix \quad \text{and} \quad C=\mat{cc} \hat C & 0\\0 & \tilde C \rix.
		\]
		Thus, in Theorem~\ref{thm:uniflmi}, if $(E,A)$ is $\Omega$-admissible, then there exists an invertible matrix $X \in \R^{n,n}$ such that $E^\top X =X^\top E \succeq 0$, and 
		\[
		\mathcal M_{\hat \Omega}(E,A,X) \prec 0
		\quad \text{and}\quad \mathcal M_{\tilde  \Omega}(E,A,X) \prec 0.%
		\]
	}
\end{remark}

The limitation of Theorem~\ref{thm:uniflmi} lies in its assumption regarding $\Omega_u$. The following theorem is a more general result that is applicable for all LMI regions lying in the left half of the complex plane. 
\begin{theorem}{\rm \cite{r5}}\label{thm:genflmi}
	Let $\Omega \subseteq \C$ be an LMI region that lies in open left half plane and let $(E,A) \in (\R^{n,n})^2$. Then $(E,A)$ is $\Omega$-admissible if and only if there exist $P,S \in \R^{n,n}$ such that 
	\begin{eqnarray*}
	&	(EP)^\top =(EP) \succeq 0, 	
    (ES)^\top =(ES) \succeq 0,	AS+(AS)^\top  \prec 0, \quad \text{and}\\
		&	M_{\Omega}(E,A,P,S):=B \otimes EP+C \otimes AP + C^\top  \otimes (AP)^\top  + I_n \otimes ES \preceq 0.
	\end{eqnarray*}
\end{theorem}

Again, we note that in Theorem~\ref{thm:genflmi}, the LMI $AS+(AS)^\top  \prec 0$ implies that the matrix $S$ is invertible as for any $x \in \C^n$ satisfying $Sx=0$, we have $(AS+(AS)^\top )x=0$.

\section{DH pairs and their relationship with $\Omega$-admissible pairs $(E,A)$}\label{sec:DHchara}

In this section, we discuss the relationship between DH pairs and $\Omega$-admissibility. The following result gives a condition on a DH pair to be $\Omega$-admissible, when $\Omega$ is an LMI region lying in the left half plane.

\begin{theorem}\label{thm:DHsuff}
Let $\Omega$ be an LMI region in the left half plane given by~\eqref{eq:lmidef} and let $(E,(J-R)Q) \in {(\R^{n,n})}^2$ be a DH matrix pair (Definition~\ref{defdhpair}) such that 
\begin{equation}\label{eq311}
	\mathcal{M}_{\Omega}(E,J,R,Q)= B \otimes Q^\top E + (C-C^\top ) \otimes Q^\top JQ -(C+ C^\top ) \otimes Q^\top RQ \prec 0.
\end{equation}
Then $(E,(J-R)Q) $ is $\Omega$-admissible. 
\end{theorem}
\begin{proof}
Let $(E,A)$ be a DH matrix pair, where $A=(J-R)Q$, $R \succ 0$, $J^\top =-J$, $Q$ invertible with $E^\top Q=Q^\top E \succeq 0$ and~\eqref{eq311} holds. In view of Theorem~\ref{thm:equiasym}, every DH matrix pair $(E,(J-R)Q)$ with $R \succ 0$ is regular and impulse free. Thus in order to show that $(E,(J-R)Q)$ is $\Omega$-admissible, we only have to show that all finite eigenvalues of $(E,(J-R)Q)$ lie inside $\Omega$-region. For this, let $\lambda$ be an eigenvalue of $(E,A)$ and let $x \in \C^n\setminus \{0\}$ be the corresponding eigenvector, that is, $(J-R)Qx=\lambda Ex$. This implies that 
\begin{equation}\label{eq:suff1}
	x^*Q^\top (J-R)Qx=\lambda x^*Q^\top Ex \quad \text{and}\quad 	-x^*Q^\top (J+R)Qx=\overline \lambda x^*Q^\top Ex,
\end{equation}
since $Q$ is invertible and $Q^\top E =E^\top Q$. Thus
\begin{align}\label{eq:suff2}
	&\left(B + \lambda C + \overline {\lambda} C^\top \right) \otimes x^{*} Q^\top  E x \nonumber \\
	&= B \otimes x^{*} Q^\top  E x + \lambda C \otimes x^{*} Q^\top  E x + \overline {\lambda} C^\top  \otimes x^{*} Q^\top  E x \nonumber\\
	&= B \otimes x^{*} Q^\top  E x + C \otimes \lambda x^{*} Q^\top  E x + C^\top  \otimes \overline {\lambda} x^{*} Q^\top  E x \nonumber \\
	&= B \otimes x^{*} Q^\top  E x + C \otimes x^{*} Q^\top  (J - R) Q x - C^\top  \otimes x^{*} Q^\top  (J + R) Q x \quad \text{[from~\eqref{eq:suff1}]} \nonumber \\
	&= B \otimes x^{*} Q^\top  E x + (C - C^\top ) \otimes x^{*} Q^\top  J Q x - (C + C^\top ) \otimes x^{*} Q^\top  R Q x \nonumber \\
	&= (I_n \otimes x)^{*} \left( B \otimes Q^\top  E + (C - C^\top ) \otimes Q^\top  J Q - (C + C^\top ) \otimes Q^\top  R Q \right) (I_n \otimes x)  \nonumber \\
	&= (I_n \otimes x)^{*} \mathcal{M}_{\Omega}(E,J, R, Q) (I_n \otimes x) \prec 0, 
\end{align}	
since  $\mathcal{M}_{\Omega}(E,J, R, Q) \prec 0$. Further, note that $x^*Q^\top Ex \neq 0$, because if $x^*Q^\top Ex=0$, then from~\eqref{eq:suff1} we have $x^*Q^\top (J-R)Qx=0$, which implies that $x^*Q^\top JQx =0$ and $x^*Q^\top RQx =0$, since $Q^\top JQ$ is skew-symmetric and $Q^\top RQ$ is symmetric. Also,  $x^*Q^\top RQx =0$ implies that $Q^\top RQx=0$, which is a contradiction  because $Q^\top RQ \succ 0$ as $Q$ is invertible and $R \succ 0$. Thus we have $x^*Q^\top Ex > 0$, since $Q^\top E \succeq 0$ and $x^*Q^\top Ex \neq 0$. Thus, in view of~\eqref{eq:suff2}, we have $\left(B + \lambda C + \overline {\lambda} C^\top \right) \otimes x^{*} Q^\top  E x \prec 0$ and $x^{*} Q^\top  E x>0$ , which implies that $(B + \lambda C + \overline {\lambda} C^\top ) \prec 0$ and therefore from~\eqref{eq:lmidef} $\lambda \in \Omega$. 
	\end{proof}

Next, we show that the converse of the above theorem holds for LMI regions $\Omega$ for which 
$\Omega_{u}$ is nonempty. 

\begin{theorem}\label{thm:iffomega1}
	Let $\Omega \in \C$ be an LMI region lying in the left half plane such that
	$\Omega_{u}= \{ z \in \mathbb{C} :~  Cz + C^{\top} \overline{z} \prec 0 \}$ is nonempty and let $(E,A)\in {(R^{n,n})}^2$ be a matrix pair. Then $(E,A) $ is $\Omega$-admissible if and only if $A=(J-R)Q$ for some $J,R,Q \in \R^{n,n}$ such that $R \succ 0$, $J^\top =-J$, $Q$ invertible with $Q^\top E =E^\top Q \succeq 0$, and $\mathcal M_{\Omega}(E,J,R,Q) \prec 0$, where the matrix  $\mathcal M_{\Omega}(E,J,R,Q)$ is defined by~\eqref{eq311}.
\end{theorem}
\begin{proof}
The ``only if" part follows from Thoerem~\ref{thm:DHsuff}. Thus we only prove the ``if part". For this, let $(E,A)$ be $\Omega$-admissible and $\Omega_u$ is nonempty. Then in view of  Theorem~\ref{thm:uniflmi} and Remark~\ref{rem:uniflminter}, there exists invertible $X \in \R^{n,n}$ such that $E^\top X=X^\top E \succeq 0$, $X^\top A+A^\top X \prec 0$ and 
\begin{equation}\label{eq:iffomega1}
	\mathcal M_{\Omega}(E,A,X):= B \otimes E^\top X + C \otimes X^\top A +C^\top  \otimes A^\top X \prec 0. 
\end{equation}
By setting 
\begin{equation*}
 Q=X, \quad 	R=-\frac{A X^{-1}+(A X^{-1})^{\top}}{2}, \quad \text{and}\quad  J=\frac{A X^{-1}-(A X^{-1})^{\top}}{2}, 
\end{equation*}
we have $J^\top =-J$, $R^\top =R$, $Q$ invertible such that $E^\top Q =Q^\top E \succeq 0$. Also, $R \succ 0$, since $X^\top A +A^\top X \prec 0$ implies that $(X^\top )^{-1}A^\top  +AX^{-1} \prec 0$  and thus $R=-\frac{A X^{-1}+(A X^{-1})^{\top}}{2} \succ 0$. Further, in view of~\eqref{eq:iffomega1}, we have
\begin{eqnarray*}
&&\mathcal M_{\Omega}(E,J,R,Q) \\
&&= B \otimes E^\top Q + (C-C^\top )\otimes Q^\top JQ - (C+C^\top ) \otimes Q^\top RQ \\
&&= B \otimes E^\top X +(C-C^\top )\otimes X^\top   \frac{A X^{-1}-(A X^{-1})^{\top}}{2}  X+  (C+C^\top )\otimes X^\top   \frac{A X^{-1}+(A X^{-1})^{\top}}{2}  X \\
&&= B \otimes E^\top X +(C-C^\top )\otimes \frac{X^\top A-A^\top X}{2} + (C+C^\top )\otimes \frac{X^\top A+A^\top X}{2} \\
&&= B \otimes E^\top X + C \otimes X^\top A +C^\top  \otimes A^\top X \prec 0.  
\end{eqnarray*} 
\end{proof}

In Theorem~\ref{thm:iffomega1}, we have provided a DH characterization of $\Omega$-admissible matrix pairs for which $\Omega_u$ is nonempty. As a result, we have a DH characterization for all nonempty uniform LMI regions $\Omega$, or equivalently, for nonempty LMI regions which are cones in the complex plane~\cite{Kus19}. For example, left half of the complex plane, conic sector regions, etc. 

However, there are regions $\Omega$, such as vertical strips, that do not satisfy the condition that $\Omega_u$ is nonempty in Theorem~\ref{thm:iffomega1}. For such regions, we have the following theorem that relaxes the assumption on $\Omega_u$ and results in a nonstrict LMI condition.  
\begin{theorem}\label{thm:generalEinvert}
	Let $\Omega$ be an LMI region that lies in the open left half of the complex plane and $(E,A)\in (\R^{n,n})^2$ such that $E$ is invertible. Then $(E,A)$ is $\Omega$-admissible if and only if $A=(J-R)Q$ for some $J,R,Q \in \R^{n,n}$ such that $J^\top =-J$, $R\succ 0$, $Q$ invertible with $E^\top Q=Q^\top E \succeq 0$ and $S \in \R^{n,n}$ invertible with $S^\top E=E^\top S \succeq 0$, and 
	\begin{equation}\label{eq:generalEinvert11}
	\mathcal M_{\Omega}(E,J,R,Q,S):=B \otimes Q^\top E + (C-C^\top )\otimes Q^\top JQ -(C+C^\top )\otimes Q^\top RQ + I_n \otimes E^\top S \preceq 0.
	\end{equation}
\end{theorem}
\begin{proof}
	First suppose that $A=(J-R)Q$ for some $J,R,Q \in \R^{n,n}$ satisfying  $J^\top =-J$, $R\succ 0$, Q invertible with $E^\top Q=Q^\top E \succeq 0$, and~\eqref{eq:generalEinvert11} holds with a matrix $S \in \R^{n,n}$ such that $S^\top E=E^\top S \succeq 0$. Let $\lambda \in \C$ be an eigenvalue of $(E,A)$ and $x\in \C^n \setminus \{0\}$ be such that $(J-R)Qx =\lambda Ex$. This implies that 
	\begin{equation}\label{proof:generalEinvert2}
		x^*Q^\top (J-R)Qx=  \lambda x^*Q^\top  Ex \quad \text{and} \quad 
			-x^*Q^\top (J+R)Qx=  \overline \lambda x^*Q^\top  Ex,
	\end{equation}	
	since $Q$ is invertible and $Q^\top E=E^\top Q$. Thus, we have
	\begin{eqnarray}\label{proof:generalEinvert3}
		&&(B +\lambda C + \overline \lambda C^\top ) \otimes x^*Q^\top Ex \nonumber\\
		&&= B \otimes x^*Q^\top Ex + \lambda C \otimes x^*Q^\top Ex + \overline \lambda C^\top  \otimes x^*Q^\top Ex  \nonumber\\
	&& =  B \otimes x^*Q^\top Ex + C \otimes x^*Q^\top (J-R)Qx -C^\top  \otimes x^*Q^\top (J+R)Qx \quad (\text{from}~\eqref{proof:generalEinvert2}) \nonumber\\
	&& =  B \otimes x^*Q^\top Ex + (C-C^\top ) \otimes x^*Q^\top JQx -(C+C^\top ) \otimes x^*Q^\top RQx \nonumber \\
		&& = (I_n \otimes x^*)\left( B \otimes Q^\top E + (C-C^\top ) \otimes Q^\top JQ- (C+C^\top ) \otimes Q^\top RQ + I_n \otimes E^\top S\right) (I_n \otimes x)  \nonumber\\
		&& \hspace{0.5cm} - (I_n \otimes x^*) (I_n \otimes E^\top S) (I \otimes x)  \nonumber\\
		&&=(I_n \otimes x^*) 	\mathcal M_{\Omega}(E,J,R,Q,S) (I_n \otimes x)- I_n \otimes x^*E^\top Sx  \prec 0,
	\end{eqnarray}
	since $\mathcal M_{\Omega}(E,J,R,Q,S) \preceq 0$ and $ I_n \otimes x^*E^\top Sx \succ 0$ as $E$ and $S$ are invertible with $E^\top S\succeq 0$.  Also, since $Q^\top E \succeq 0$, we have $x^* Q^\top Ex > 0$, as if  $x^* Q^\top Ex=0$, then from~\eqref{proof:generalEinvert2}, we have $x^* QRQx=0$ and thus $RQx=0$, which is a contradiction as $Q$ is invertible and $R \succ 0$. 
	Thus, from~\eqref{proof:generalEinvert3}, we have $(B +\lambda C + \overline \lambda C^\top ) \otimes x^*Q^\top Ex \prec 0$ and $ x^*Q^\top Ex >0$ implies that $B +\lambda C + \overline \lambda C^\top  \prec 0$. This implies from~\eqref{eq:lmidef} that $\lambda \in \Omega$. Further, $(E,A)$ is a DH matrix pair with $R \succ 0$, thus in view of Theorem~\ref{thm:equiasym}, we have that $(E,A)$
is $\Omega$-admissible. 
	
Conversely, let $(E,A)$ is $\Omega$-admissible. Then in view of Theorem~\ref{thm:genflmi} and Remark~\ref{rem:uniflminter}, there exists $P, S\in \R^{n,n}$ such that 
\begin{equation}\label{proof:generalEinvert44}
E^\top P=P^\top E \succeq 0,\quad E^\top S=S^\top E \succeq 0, \quad A^\top S+S^\top A \prec 0,\quad  P^\top A+A^\top P+E^\top S \preceq 0,
\end{equation}
and 
\begin{equation}\label{proof:generalEinvert45}
B \otimes E^\top P + C\otimes P^\top A +C^\top  \otimes A^\top P + I \otimes E^\top S \preceq 0.
\end{equation}
Note that $S$ is invertible, since $A^\top S +S^\top A \prec 0$, and thus $E^\top S \succ 0$, since $E^\top S \succeq 0$ and $E$ is assumed to be invertible. Thus from~\eqref{proof:generalEinvert44}, we have that 
\[
P^\top A+A^\top P \prec 0.
\]
By setting, 
\begin{equation}
	Q=P, \quad 	R=-\frac{A P^{-1}+(A P^{-1})^{\top}}{2}, \quad \text{and}\quad  J=\frac{A P^{-1}-(A P^{-1})^{\top}}{2}, 
\end{equation}
we have $J^\top =-J$, $R^\top =R$, $Q$ invertible such that $E^\top Q =Q^\top E \succeq 0$. Also, $R \succ 0$, since $P^\top A +A^\top P \prec 0$  and $A=(J-R)Q$. Further, we have that
\begin{eqnarray*}
	&&\mathcal M_{\Omega}(E,J,R,Q,S)\\
    &&=B \otimes Q^\top E + (C-C^\top )\otimes Q^\top JQ -(C+C^\top )\otimes Q^\top RQ + I_n \otimes E^\top S \\
	&&= B \otimes Q^\top E + (C-C^\top )\otimes \frac{Q^\top A-A^\top Q}{2} 
	-(C+C^\top )\otimes \frac{Q^\top A+A^\top Q}{2} + I_n \otimes E^\top S  \\
	&&=B \otimes Q^\top E + C\otimes Q^\top A +C^\top  \otimes A^\top Q + I_n \otimes E^\top S	\preceq 0.
	\end{eqnarray*} 
\end{proof}

The ``only if" part of Theorem~\ref{thm:generalEinvert} also holds for singular $E$ when $S$ is replaced by $Q$, as shown in the following result. 
\begin{theorem}
	Let $\Omega$ be an LMI region that lies in the open left half of the complex plane and $(E,(J-R)Q)\in (\R^{n,n})^2$ be a DH matrix pair, that is, 
	\[
	J^\top =-J,\quad R\succ 0, \quad Q~\text{invertible} \quad \text{and} \quad Q^\top E=E^\top Q \succeq 0.
	\]
	Then $(E,(J-R)Q)\in (\R^{n,n})^2$ is $\Omega$-admissible if 
	\begin{equation*}
		B \otimes EQ^{-1} + (C-C^\top )\otimes J -(C+C^\top )\otimes R + I_n \otimes EQ^{-1} \preceq 0.
	\end{equation*}
\end{theorem}
\begin{proof}
The proof is similar to the proof of the ``only if" part of Theorem~\ref{thm:generalEinvert}.
\end{proof}


\section{Special LMI regions and $\Omega$-stability}\label{sec:specialomega}

We have discussed, in Section~\ref{sec:DHchara}, relationship between DH pairs and $\Omega$-admissibility, where $\Omega$ is an LMI region inside left half of the complex region. If the $\Omega$-region intersects with the right half, then the results of Section~\ref{sec:DHchara} do not hold in general, as their proofs use the definiteness of the dissipation matrix $R$ which ensures the impulsefreeness and enforces all eigenvalues of $(E,(J-R)Q)$ to the left half of the complex plane.

In this section, we consider some special LMI regions $\Omega$ that are allowed to intersect with the right half of the complex plane and derive conditions on $\Omega$-stable matrix pairs $(E,(J-R)Q)$ by allowing $R$ to be indefinite. These regions include  left and right conic sectors,  disks centered on the real line, vertical left and right halfplanes, ellipsoid centered on the real line, left and right parabolic regions centered on the real line, left and right hyperbolas with vertices on the real line, and horizontal strip. 

 With the exception of the definiteness constraint on $R$, we derive constraints on the matrix pairs $(E, (J-R)Q)$ that have all their eigenvalues inside these particular regions. Although their proofs differ from the proof of Theorem~\ref{thm:DHsuff}, the results are comparable to it. This section adopts the notation in~\cite{ChouGS24} for these specific LMI regions and is an extension of Section~\ref{sec:DHchara} in~\cite{ChouGS24} for matrix pairs.

Let us show an example for the disks centered on the real line, that is, 
  let us consider a disk centered at $(q,0)$ with radius $r>0$, denoted by $\Omega_D(q,r)$, and defined as
\[
\Omega_D(q,r)
:= \left\{ z\in \C \ : \ |z-q|<r\right\}.
\] 
The disk region $\Omega_D(q,r)$ can be characterized in form of~\eqref{eq:lmidef} of an LMI region with matrices
\[
B=\mat{cc} -r & q \\q &-r \rix \quad \text{and}\quad 
C=\mat{cc}0&0\\-1&0\rix.
\]
The following result derives a condition for $\Omega_D$-stability.
\begin{theorem}\label{thm:omegastable}
Let $q \in \R$ and $r>0$, and consider the disk region $\Omega_D(q,r)$. Let $(E,(J-R)Q) \in (\R^{n,n})^2$ be a regular  matrix pair with $J^\top= -J$, $R^\top =R$, and $Q$ invertible with $Q^\top E=E^\top Q \succeq 0$ such that 
\begin{equation}\label{thm:stabcond1}
    \mat{cc} rEQ^{-1} & qEQ^{-1}\\ qEQ^{-1} & rEQ^{-1}\rix  \succ
\mat{cc} 0 & J-R \\(J-R)^\top  & 0\rix.
\end{equation}
Then $(E,(J-R)Q)$ is $\Omega_D$-stable.
\end{theorem}
\begin{proof}
Consider $(E,(J-R)Q$, where $J^\top =-J$, $R^\top =R$, and $Q$ invertible with $Q^\top E=E^\top Q \succeq 0$ such that~\eqref{thm:stabcond1} holds. Let $\lambda \in \C$ be an eigenvalue of $(E,(J-R)Q)$ and $v \in \C^n \setminus \{0\}$ be a corresponding eigenvector, that means
\begin{eqnarray}\label{prof:stabcond1}
    v^*(J-R)Q=\lambda v^*E.
\end{eqnarray}
Since~\eqref{thm:stabcond1} holds, we have that 
\[
\mat{cc}v^* & 0 \\ 0 & v^* \rix
\mat{cc}rEQ^{-1} & q EQ^{-1} \\ q EQ^{-1} & r EQ^{-1}\rix
\mat{cc}v & 0 \\ 0 & v \rix 
~\succ~
\mat{cc}v^* & 0 \\ 0 & v^* \rix
\mat{cc}0 & J-R \\ (J-R)^\top  & 0\rix
\mat{cc}v & 0 \\ 0 & v \rix. 
\]
This implies that 
\begin{eqnarray}\label{prof:stabcond2}
    \mat{cc}r & q \\ q & r\rix (v^*EQ^{-1}v) ~\succ~ 
    \mat{cc}0 & v^*(J-R)v \\ v^*(J-R)^\top v & 0\rix.
\end{eqnarray}
Note that $v^*EQ^{-1}v  \neq 0$, since $E^\top Q \succeq 0$ implies that $EQ^{-1} \succeq 0$. In fact, if $v^*EQ^{-1}v=0$, then we have $v^*EQ^{-1} = 0$ as $EQ^{-1} \succeq 0$ and also from~\eqref{prof:stabcond1} $v^*(J-R)=0$. Thus $v$ becomes a nonzero vector in the common left null space of $E$ and $(J-R)Q$. This implies that $(E,(J-R)Q$ is singular, which is a contradiction. Thus we have $v^*EQ^{-1}v  \neq 0$ and, in particular, $v^*EQ^{-1}v  > 0$ as $EQ^{-1} \succeq 0$. Thus from~\eqref{prof:stabcond2}, we obtain that 
\[
    \mat{cc}r & q \\ q & r\rix \succ 
    \mat{cc}0 & \frac{v^*(J-R)v}{v^*EQ^{-1}v} \\ \frac{v^*(J-R)^\top v}{v^*EQ^{-1}v} & 0\rix.
\]
and in view of Lemma~\ref{lem2}, we have 
\[
    \mat{cc}r & q-\lambda \\ q- \overline \lambda  & r\rix \succ 0.
\]
This implies that $\lambda \in \Omega_D(q,r)$ and hence $(E,(J-R)Q)$ is $\Omega_D$-stable.
\end{proof}

A result similar to Theorem~\ref{thm:omegastable} can also be obtained for other regions in~\cite[Section~3]{ChouGS24}. 
In Table~\ref{tab:specialdhpairlmi}, we provide a sufficient condition for matrix pairs $(E,(J-R)Q)$ with eigenvalues in these regions. 

\begin{table}[ht]
    \centering
\begin{tabular}{|c|c|}
\hline
\textbf{LMI Region} & 
\begin{tabular}{c}
\textbf{Constraints on $\Big. J^\top =-J \Big.$, $R^\top =R$,} \\
\textbf{$\Big. Q \Big.$ with $E^\top Q \succeq 0$, and $T := EQ^{-1}$}
\end{tabular} 
\\ \hline
Left Conic Sector: $\Omega_{C_L}(a,\theta)$ 
& $\mat{cc} \sin(\theta) (aT+R) & -\cos(\theta) J \\ \cos(\theta) J & \sin(\theta) (aT+R) \rix \succ 0 $\\ \hline
Right Conic Sector: $ \Omega_{C_R}(a,\theta)$ & $\mat{cc} -\sin(\theta) (aT+R) & -\cos(\theta) J \\ \cos(\theta) J & -\sin(\theta) (aT+R) \rix \succ 0 $ \\ \hline
Disk: $\Omega_D(q,r)$ & $\mat{cc} rT & qT-J+R \\ qT+J+R & rT
\rix \succ 0 $ \\ \hline
Vertical strip: $ \Omega_V(h,k)$ & $\mat{cc} kT+R & 0 \\ 0 & -hT-R
\rix \succ 0$\\ \hline
Left halfplane: $ \Omega_V(-\infty,k)$ & $ kT+R \succ 0$ \\ \hline
Right halfplane: $ \Omega_V(h,\infty)$ & $ -hT-R \succ 0$\\ \hline
Ellipsoid: $ \Omega_E(q_e,a_e,b_e)$ & $\mat{cc} a_eT & q_eT -  \frac{a_e}{b_e} J + R \\ q_eT +  \frac{a_e}{b_e} J + R & a_eT\rix \succ 0 $ \\ \hline
Left parabolic region: $ \Omega_{P_L}(q_p,c_p)$ & $\mat{cc} T & -\sqrt{\frac{c_p}{2}}J \\ \sqrt{\frac{c_p}{2}}J  & q_pT+ R \rix \succ 0$ \\ \hline
Right parabolic region: $ \Omega_{P_R}(q_p,c_p)$ & $\mat{cc} T & -\sqrt{\frac{c_p}{2}}J \\ \sqrt{\frac{c_p}{2}}J  & -q_pT- R \rix \succ 0$\\ \hline
Left hyperbola: $ \Omega_{Hyp,L}(a_h,b_h)$ & $ \mat{cc} \frac{R}{a_h} & -T - \frac{J}{b_h}\\ -T+\frac{J}{b_h} & \frac{R}{a_h} \rix \succ 0$ \\ \hline
Right hyperbola: $\Omega_{Hyp,R}(a_h,b_h)$ & $ \mat{cc} -\frac{R}{a_h} & -T - \frac{J}{b_h}\\ -T+\frac{J}{b_h} & -\frac{R}{a_h} \rix \succ 0$ \\ \hline
Horizontal strip: $\Omega_H(w)$ & $\mat{cc} wT & -J \\ J & wT 
\rix \succ 0$\\ \hline
\end{tabular}
    \caption{$\Omega$-stability conditions on pairs $(E,(J-R)Q)$ for specific regions mentioned in Section 3 of~\cite{ChouGS24}.  }
    \label{tab:specialdhpairlmi}
\end{table}

\section{The nearest $\Omega$-admissible pair problem in DH form}\label{sec:omegastable}

Given an LMI region $\Omega$ and a pair $(E,A)$, we now want to tackle Problem~\eqref{eqa1}, that is, look for the pair $(\tilde E, \tilde A)$ that is $\Omega$-admissible and that is the nearest to $(E,A)$ in the Frobenius norm: 
\begin{equation} \label{eqa1v2}
\inf_{(\tilde{E},\tilde{A}) \in \mathcal{S}_{\Omega}^a}{{\|A-\tilde{A}\|}_{F}^{2} + \mu {\|E-\tilde{E}\|}_{F}^{2}}, 
\end{equation} 
where we have introduce a penalty parameter $\mu > 0$ that allows one to balance the importance of $A$ and $E$ in the objective.  
To solve~\eqref{eqa1v2}, we define the set 
$\mathcal{S}_{\Omega dh}^a$ containing all DH pairs $(\tilde E, (J-R)Q)$ such that 
\mbox{$\mathcal M_{\Omega}(\tilde E,J,R,Q) \prec 0$}, where $\mathcal M_{\Omega}(\tilde E,J,R,Q)$ is defined by~\eqref{eq311}, that is, 
\begin{equation}
    \mathcal{S}_{\Omega dh}^a:=
    \{(\tilde E, (J-R)Q) \in (\R^{n,n})^2:~J^{\top}=-J, R \succ 0, Q^\top \tilde E  \succeq 0, M_{\Omega}(\tilde E,J,R,Q) \prec 0
    \}.
\end{equation}
Let us introduce a new variable $T=\tilde E Q^{-1}$. Note that, for any invertible $Q$, 
\[
T = \tilde E Q^{-1}  \succeq 0 \quad 
\iff \quad  
Q^{\top} \tilde E \succeq 0, 
\] 
and also 
\[
M_{\Omega}(\tilde E,J,R,Q) \prec 0 \iff \hat M_{\Omega}(T,J,R) \prec 0,
\]
where 
\begin{eqnarray}
    \hat M_{\Omega}(T,J,R):= B \otimes T + (C-C^\top)\otimes J
    -(C+C^\top)\otimes R.
\end{eqnarray}
This implies that the set $\mathcal{S}_{\Omega dh}^a$ can be equivalently written as 
\[
\mathcal{S}_{\Omega dh}^a =   \{(TQ, (J-R)Q) \in (\R^{n,n})^2:~J^{\top}=-J, R \succ 0, T  \succeq 0, Q\,\text{invertible},\, \hat M_{\Omega}(T,J,R) \prec 0
    \}=:\hat {\mathcal S}_{\Omega dh}^a.
\]
In view of Theorem~\ref{thm:iffomega1}, if $\Omega$ is a nonempty uniform LMI region, then the set of all $\Omega$-admissible pairs can be characterized as the set $\mathcal{S}_{\Omega dh}^a$ of $\Omega$-admissible DH pairs, that is,  $\mathcal S_\Omega^a=\mathcal{S}_{\Omega dh}^a$.
Thus for uniform LMI regions, the problem~\eqref{eqa1v2} can be reformulated as the problem of finding the nearest $\Omega$-admissible DH pair from $\hat {\mathcal S}_{\Omega dh}^a$ as follows:
\begin{align} \label{eq:paramV2}
\min_{T \succeq 0, Q \text{ invertible}, J^\top = -J, R\succ 0}  & \quad 
{\| A - (J-R)Q\|}_F^2 + 
\mu {\|E - TQ \|}_F^2  \nonumber \\ 
\text{ such that } &  \quad   
(T,J,R) ~\text{ satisfies}~  \hat M_{\Omega}(T,J,R) \prec 0.  \tag{$\mathcal P_{\Omega dh}^a$}
\end{align}

For general LMI region $\Omega$ lying in the left half plane, in view of Theorem~\ref{thm:DHsuff}, we have that 
$\hat {\mathcal S}_{\Omega dh}^a  \subseteq \mathcal S_\Omega^a$. This implies that in such situations, a solution to the problem~\eqref{eq:paramV2} gives an approximate solution to the problem~\eqref{eqa1v2}.

If $\Omega$ intersects with the right half, an approximate solution to the $\Omega$-admissibility problem~\eqref{eqa1v2} is achieved by solving~\eqref{eq:paramV2} 
without the definiteness constraint on $R$. This is due to  Table~\ref{tab:specialdhpairlmi}, where the pair $(TQ,(J-R)Q)$ with $J^{\top}=-J$, $R^{\top} =R$, $T  \succeq 0$, $Q$ being invertible, and $\hat M_{\Omega}(T,J,R) \prec 0$, implies that $\Omega$-stable. Additionally, in all our numerical experiments, the final solution $(\tilde E,\tilde A)$ is found to be impulse-free, that is, $(\tilde E,\tilde A)$ has $\text{rank}(\tilde E)$ number of finite eigenvalues. By solving~\eqref{eq:paramV2} without a definiteness constraint on $R$, we obtain an approximate solution to the $\Omega$-stability problem that is also impulse-free and, hence, $\Omega$-admissible.


In this section, we propose two algorithms to solve~\eqref{eq:paramV2}:   
\begin{itemize}
    \item One specialized for the case of Hurwitz stability in Section~\ref{sec:algohur}. 

    \item One for the general problem~\eqref{eq:paramV2} in Section~\ref{sec:algoBCD}. 
\end{itemize}

\subsection{Fast gradient method for Hurwitz stability} \label{sec:algohur}

For the case of Hurwitz stability, \eqref{eq:paramV2} reduces to 
\begin{equation} \label{eq:paramHur} 
\min_{T \succeq 0, Q \text{ invertible}, J^\top = -J, R}    
{\| A - (J-R)Q\|}_F^2 + 
\mu {\|E - TQ \|}_F^2  
\quad
\text{ such that } 
\quad 
R \succeq 0. 
\end{equation}
To tackle~\eqref{eq:paramHur}, we adapt the algorithm from~\cite{r1}. 
In \cite{r1}, only Hurwitz stability was considered, and the authors used another similar parametrization, namely $H := \tilde E^\top Q$, leading to the optimization problem 
\begin{equation} \label{eq:Volker}
\min_{H \succeq 0, Q \text{ invertible}, J^\top = -J, R\succeq 0} 
{\| A - (J-R)Q\|}_F^2 + 
\mu {\|E^\top - H Q^{-1} \|}_F^2. 
\end{equation}
This parametrization involves the inverse of matrix $Q$ in the objective, and hence is numerically harder and more unstable to optimize. 
In fact, we will show that our new parametrization~\eqref{eq:paramHur}  outperforms the one in~\cite{r1}; see Section~\ref{sec:hurwitznumexp} for numerical experiments.   
 Moreover, and this is crucial, the parametrization  $H := \tilde E^\top Q$ would not lead to convex LMIs as $T = H Q^{-1}$ appears in the constraints and is not linear in $(H,Q)$ while it is linear in $T$. 

The algorithm from~\cite{r1} is a projected fast gradient method, that is, a projected gradient method with extrapolation that accelerates convergence. The projection onto the feasible set only requires the projection onto the PSD cone for the variables $T$ and $R$, which can be computed via an eigenvalue decomposition in $O(n^3)$ operations~\cite{Hig88b}. The adaptation of the algorithm from~\cite{r1} to tackle~\eqref{eq:Volker} was straightforward as the structure of both problems, \eqref{eq:Volker} and~\eqref{eq:paramHur}, is very similar (we replaced the variable $H$ by the variable $T$ in the code, and adapted the gradient and objective computation). Hence we do not provide further details here.

\paragraph{Initialization} To initialize $(J,R,T,Q)$, we rely on the same strategy as in~\cite{r1}: set $Q= I_n$, in which case the optimal $T$ is the projection onto the PSD cone of $E$, the optimal $J$ is the skew-symmetric part of $A$, that is, $(A-A^\top)/2$, and the optimal $R$ is the projection onto the PSD cone of the symmetric part of $A$, that is, $(A+A^\top)/2$.

\subsection{Block coordinate descent in the general case}   \label{sec:algoBCD}

When the feasible set does not involve PSD matrices as in the Hurwitz case, projecting onto the feasible set is more tricky, and we resort to the interior point method SDPT3~\cite{toh1999sdpt3,tutuncu2003solving} with CVX as a modeling system~\cite{cvx,gb08}. 
Moreover, we have experienced that gradient-based methods was not working as well as a simple block coordinate descent (BCD) method, because the cost of the projection is essentially the same as that of optimizing over a subset of variables. 
More precisely, we optimize over the variable $Q$ and the variables $(T,J,R)$ alternatively:  
\begin{itemize}

\item for $(T,J,R)$ fixed, the optimal solution for $Q$ is given by solving an unconstrained least squares problem: 
\[
\min_Q {\| A - (J-R)Q\|}_F^2 + 
\mu {\|E - T Q \|}_F^2. 
\]

\item For $Q$ fixed, the optimal solution for $(T,J,R)$ requires to solve the following semidefinite program: 
\begin{eqnarray} \label{Qfixed}
\min_{T \succeq 0, J^\top=-J, R^\top=R}    
{\| A - (J-R)Q\|}_F^2 + 
\mu {\|E - TQ \|}_F^2   \\
\text{ such that } 
\quad (T,J,R) \text{ satisfies LMI from Table~\ref{tab:specialdhpairlmi}.} \nonumber
\end{eqnarray}

\end{itemize}

We have added an extrapolation step to this BCD scheme that allows one to accelerate convergence~\cite{phan2023inertial}.  

To initialize $(J,R,T,Q)$, we rely on the same strategy as in the previous section: set $Q= I_n$, and then obtain the optimal $(T,J,R)$ by solving~\eqref{Qfixed}.

\section{Numerical Experiments}

All experiments are run on a laptop 12th Gen Intel(R) Core(TM) i9-12900H (2.50 GHz), 32Go of RAM. The code is available from \url{https://gitlab.com/ngillis/nearest-omega-stable-pair}.  

To the best of our knowledge, there exists only a few previous works to compute nearest stable matrix pairs: 
\begin{itemize}
    \item The paper~\cite{r1} focus on the Hurwitz stability. We will refer to this approach as ``old DH'' as it is an older method based on another parametrization; see~\eqref{eq:Volker}. 

    \item The paper~\cite{GilKS20} focuses on the Schur stability,  under rather strict rank constraints, hence we do not consider it here. 

    \item Noferini and Nyman~\cite{noferini2025finding} rely on the generalized Schur form: given a matrix pair $(E,A)$, it is possible to compute (complex) orthogonal matrices $(Q,Z)$ such that $(E_t, A_t) = (Q E Z, Q A Z)$ is an upper triangular pair. 
The multiset of the finite and infinite eigenvalues (counted with algebraic
multiplicity) of a regular upper triangular pair $(E_t, A_t)$ is given by 
$\{-A_t(i,i)/E_t(i,i) \}_{i=1,\dots,n}$. Therefore, these entries can be restricted to a specified set. However, projection onto that set is not straightforward. Authors develop code for Hurwitz and Schur stability, optimizing over the set of orthogonal matrices using Manopt~\cite{manopt}; see~\url{https://github.com/NymanLauri/nearest-stable-pencil}. We will refer to this approach as ``Manopt''. 

\end{itemize}

This paper presents the first algorithm for computing nearest $\Omega$-stable pairs in contexts beyond Schur and Hurwitz, as detailed in Table~\ref{tab:specialdhpairlmi}. Such examples will be provided in Section~\ref{sec:neelam}; however, we will first compare our proposed approach to those of~\cite{r1} and~\cite{noferini2025finding}.  

We will refer to our approach as ``new DH''.

\subsection{Hurwitz stability} \label{sec:hurwitznumexp}

Let us start with Hurwitz stability where we use two types of matrices: 
\begin{itemize}
    \item Grcar~\cite{grcar1989operator}. The matrix $A$ is the Grcar matrix of order $k$ which is a banded Toeplitz
matrix with its subdiagonal set to $-1$ and both its main
and $k$ superdiagonals set to 1, while $E = I_n$.  
    All eigenvalues of $(E,A)$ have a positive real part. These matrices were used in~\cite{gillis2017computing, guglielmi2017matrix, noferini2025finding} for similar numerical experiments.  

    \item Mass-spring-damper (MSD)~\cite[Example 1.1]{MehMS16}. This is a specific descriptor dynamical system: 
    \begin{equation} \label{msdmat}
E =  \mat{cc} M & 0\\ 0  &  I_{n} \rix,~
A = (J-R) Q,~
J =   \mat{cc} 0 & -I_{n} \\ I_{n} & 0 \rix ,~
R =  \mat{cc} D & 0 \\ 0 & 0 \rix ,~
Q =  \mat{cc} I_{n} & 0 \\ 0 & K \rix, 
\end{equation} 
where $M \succ 0$ is the mass matrix. 
    To make it unstable, $R$ is replaced with $\mat{cc} D & 0 \\ 0 & -\epsilon I_n \rix$ for some parameter $\epsilon > 0$; the larger $\epsilon$, the more unstable is the system. 
    These matrices were used in~\cite{r1, noferini2025finding} for similar numerical experiments.  

\end{itemize}

Table~\ref{tab:hurwitz} reports the final relative error of the solutions produced by the various algorithms: given an approximation  $(\tilde E, \tilde A)$ for $(E,A)$, it is defined as 
\begin{equation} \label{eq:relerr}
    \text{ relative error } \; = \; 
\sqrt{
\frac{
\|A - \tilde{A}\|_F^2 + \|E - \tilde{E}\|_F^2
}
{
\|A \|_F^2 + \|E\|_F^2 
}
}  , 
\end{equation}
which should be between 0 and 1, as the trivial solution, $\tilde{A} = \tilde{E} = 0$, provides an error of 1. 
Each algorithm is given $3n$ seconds to run, which allows all of them to converge for these matrices; see Figures~\ref{fig:grcarHurwitz} and~\ref{fig:MSDHurwitz} in Appendix~\ref{app:figs}. 
 \begin{center}
\begin{table}[h!]
\begin{center}  
\begin{tabular}{|c||c|c|c|}
\hline 
  &   old DH~\eqref{eq:Volker} \cite{r1}  & Manopt \cite{noferini2025finding} & new DH~\eqref{eq:paramV2} \\ \hline 
Grcar ($n=10, k=1$) & 31.53   & \textbf{23.40}    & 31.53    \\
Grcar ($n=10, k=2$) & 26.98   & \textbf{21.57}    & 22.50    \\
Grcar ($n=10, k=3$) & 22.48   & \textbf{20.05}   & 20.87    \\
Grcar ($n=20, k=1$) & 30.87   & \textbf{22.77}   & 30.87    \\
Grcar ($n=20, k=2$) & 27.20   & \textbf{14.49}    & 23.42    \\
Grcar ($n=20, k=3$) & 23.56   & \textbf{13.26}    & 17.69    \\
Grcar ($n=30, k=1$) & 30.64   & \textbf{18.46}    & 30.64    \\
Grcar ($n=30, k=2$) & 27.93   & \textbf{11.65}    & 23.63    \\
Grcar ($n=30, k=3$) & 23.84   & \textbf{10.55}    & 19.00    \\
\hline 
MSD ($n=10, \epsilon=0.01$) & $2.4 \cdot 10^{-6}$   
& 0.056    
& $\mathbf{4.1 \cdot 10^{-8}}$    \\
MSD ($n=10, \epsilon=0.05$) & 2.59   & \textbf{1.45}    & \textbf{1.45}    \\
MSD ($n=10, \epsilon=0.10$) & 8.23   & \textbf{1.45}    & 2.06    \\
MSD ($n=20, \epsilon=0.01$) & 3.55   & 0.75    & \textbf{0.56}    \\
MSD ($n=20, \epsilon=0.05$) & 6.09   & \textbf{0.98}    & 1.15    \\
MSD ($n=20, \epsilon=0.10$) & 7.43   & \textbf{0.95}    & 1.41    \\
MSD ($n=30, \epsilon=0.01$) & 4.48   & 0.59    & \textbf{0.58}    \\
MSD ($n=30, \epsilon=0.05$) & 5.44   & \textbf{0.75}    & 0.93    \\
MSD ($n=30, \epsilon=0.10$) & 7.09   & \textbf{0.75}    & 1.13    \\\hline 
\end{tabular} 
\caption{Relative error for different algorithms for Hurwitz stability, with a time limit of $3n$ seconds. The best result is highlighted in bold. 
\label{tab:hurwitz}}   
\end{center}
\end{table}
\end{center} 
We observe the following: 
\begin{itemize}
    \item The formulation~\eqref{eq:Volker} (old DH) performs significantly worse than~\eqref{eq:paramHur} (new DH); this confirms our intuition that using the inverse of $Q$ as an optimization variables makes the problem harder to tackle.

    \item The Manopt-based approach from~\cite{noferini2025finding} performs on average better than new DH method; however, this is not consistently the case, particularly in systems nearing stability (MSD matrices with $\epsilon = 0.01$). The explanation is provided in~\cite{noferini2025finding}: when a system is far from being stable, its nearest stable approximation is likely to have large Jordan chains which are hard to numerically converge to, but Manopt avoids this pitfall by working with the Schur form directly.   

\item The computational costs of Manopt and the new DH are similar; see Figures~\ref{fig:grcarHurwitz} and~\ref{fig:MSDHurwitz} in Appendix~\ref{app:figs} that display the evolution of the relative error as a function of time. Both algorithms scale in $O(n^3)$ operations per iteration. 

    

\end{itemize}

\subsection{Schur stability} 

We now consider Schur stability, that is, all finite eigenvalues with modulus smaller than one. We use two types of matrices: 
\begin{itemize}
    \item Grcar~\cite{grcar1989operator} as for the Hurwitz case,  with $E=I_n$. Note that these matrices are far from being Schur stable. 

    \item Near-Schur stable pairs: we generate $A$ as a $n$-by-$n$ random orthogonal matrix, taking the $Q$ factor in a QR factorization of a random Gaussian matrix. Then we generate a random Gaussian $N$, and set $A = A + \epsilon N/\|N\|_F \sqrt{n}$ where $\sqrt{n} = \|A\|_F$. We take $E = I_n$. 

\end{itemize}
Since $E = I_n$ in both cases, we can use the method from Choudhary et al.~\cite{ChouGS24}, which we refer to as Choudhary, to find the nearest Schur stable matrix $\tilde A$ from $A$, while $E$ is not touched. 

Table~\ref{tab:schur} reports the relative error of the solutions generated by the different algorithms. 
Each algorithm is given up to $1000$ seconds to run; 
 see Figures~\ref{fig:grcarSchur} and~\ref{fig:nearSchur} in Appendix~\ref{app:figsSchur} showing the evolution of the relative errors over time.  
 \begin{center}
\begin{table}[h!]
\begin{center}  
\begin{tabular}{|c||c|c|c|}
\hline 
  &   Choudary~\cite{ChouGS24}  & Manopt \cite{noferini2025finding} & new DH~\eqref{eq:paramV2} \\ \hline 
Grcar ($n=10, k=1$) & 38.27    & \textbf{23.40}   & 27.06  \\
Grcar ($n=10, k=2$) & 39.40    & \textbf{20.69}   & 24.19  \\
Grcar ($n=10, k=3$) & 40.85    & \textbf{18.29}   & 20.19  \\
Grcar ($n=20, k=1$) & 39.24    & \textbf{16.10}   & 27.75  \\
Grcar ($n=20, k=2$) & 44.06    & \textbf{14.49}   & 21.83  \\
Grcar ($n=20, k=3$) & 48.90    & \textbf{12.78}   & 20.98  \\
Grcar ($n=30, k=1$) & 39.55    & \textbf{18.46}   & 27.97  \\
Grcar ($n=30, k=2$) & 45.42    & \textbf{11.63}   & 23.47  \\
Grcar ($n=30, k=3$) & 51.21    & \textbf{10.51}   & 23.48  \\ 
\hline  
Near-Schur ($n=10, \epsilon=0.01$) & 0.07   & \textbf{0.03}   & 0.05  \\
Near-Schur ($n=10, \epsilon=0.10$) & 0.59   & \textbf{0.34}   & 0.35  \\
Near-Schur ($n=10, \epsilon=1.00$) & 9.96   & \textbf{4.32}   & 5.58  \\
Near-Schur ($n=20, \epsilon=0.01$) & 0.13   & \textbf{0.04}   & 0.08  \\
Near-Schur ($n=20, \epsilon=0.10$) & 1.04   & \textbf{0.50}   & 0.65  \\
Near-Schur ($n=20, \epsilon=1.00$) & 8.27   & \textbf{2.88}   & 3.69  \\
Near-Schur ($n=30, \epsilon=0.01$) & 0.08   & \textbf{0.04}   & 0.06  \\
Near-Schur ($n=30, \epsilon=0.10$) & 0.81   & \textbf{0.38}   & 0.39  \\
Near-Schur ($n=30, \epsilon=1.00$) & 12.59   & \textbf{3.15}   & 5.49  \\
\hline 
\end{tabular} 
\caption{Relative error for different algorithms for Schur stability, with a time limit of $10n$ seconds. The best result is highlighted in bold.  
\label{tab:schur}}   
\end{center}
\end{table}
\end{center} 
We observe the following: 
\begin{itemize}
    \item Since the approach from~\cite{ChouGS24} only modifies $A$, it is expected that it performs worse than the other two, as there are less degrees of freedom to approximate the pair $(E,A)$. 

    \item The Manopt-based approach from~\cite{noferini2025finding} performs better than new DH, although new DH  performs similarly, in terms of final relative errors, when the given system is closer to a stable one, that is, for near-Schur matrix pairs (in 5 cases, the relative errors are less than 0.05\% apart, when $\epsilon \leq 0.1$). 
    This is the same reason as for the Hurwitz case.

\item Manopt is significantly faster than new DH and Choudhary; see Figures~\ref{fig:grcarSchur} and~\ref{fig:nearSchur} in Appendix~\ref{app:figsSchur} that display the evolution of the relative error as a function of time. 
The reason is that new DH and Choudhary need to solve, at each iteration, semidefinite programs with $O(n^2)$ variables, which scales roughly as $O(n^6)$ operations per iteration, 
as opposed to Manopt that scales in $O(n^3)$ operations per iteration. A direction of further study is to design faster algorithms to tackle~\eqref{eq:paramV2}. 

    

\end{itemize}

\subsection{Other sets $\Omega$: examples from~\cite{ChouGS24}} \label{sec:neelam} 

We now present two examples form~\cite{ChouGS24} with other sets $\Omega$. To the best of our knowledge, our algorithm is the only one in the literature able to solve the nearest $\Omega$ stable pair in such cases. 

The code in~\cite{ChouGS24} allows one to generate  randomly $n$-by-$n$ matrices $A$ with eigenvalues in a predefined set, and then add some Gaussian noise controlled by the parameter $\epsilon$. We set $E= I_n$, so that the eigenvalues of $(E,A)$ coincide with that of $A$.

\begin{example}

Let us consider $\Omega$ as the intersection of 
\begin{itemize}
    \item a vertical strip between -5 and 5, 
    \item a horizontal strip between -3 and 3,  
    \item  a left parabolic region centered at $(6,0)$ and curvature 1, and 
    \item  a right parabolic region centered at $(-6,0)$ and curvature 1.   
\end{itemize} 

As in~\cite{ChouGS24}, we use\footnote{We also use the random seed 2017 as in~\cite{ChouGS24} to make the experiments reproducible.} $n=10$ and $\epsilon = 1$. 
The algorithm of~\cite{ChouGS24} that imposes $\tilde E = I_n$ provides an $\Omega$-stable matrix $\tilde A$ with relative error  $\frac{\|A - \tilde{A} \|_F}{\|A\|_F} = 18.1\%$, with $\|A-\tilde A\|_F^2 = 14.05$. 
Our algorithm approximates $(E,A)$ with $(\bar E, \bar A)$
with a relative error~\eqref{eq:relerr} of $4.3\%$, 
with $\|A-\bar A\|_F^2 + \| E- \bar E\|_F^2 = 0.81$. 
This shows that allowing to modify $E$ can reduce the approximation error significantly. 
Figure~\ref{fig:ex1} illustrates the set $\Omega$ along with the eigenvalues of these decompositions.  
\begin{figure}[ht!]
\begin{center}
\includegraphics[width=0.6\textwidth]{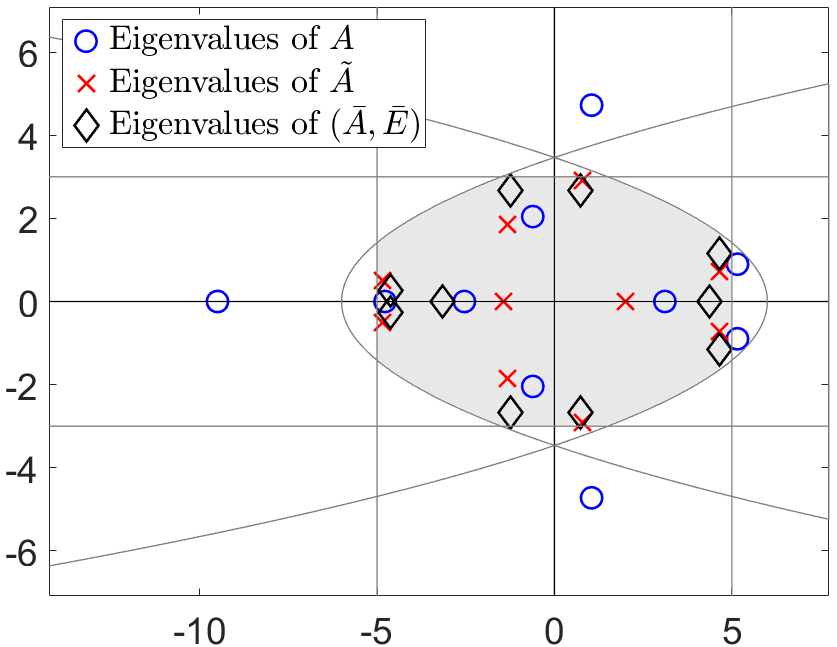}
\caption{Eigenvalues of
$A$, of its $\Omega$-stable approximation $\tilde{A} = 
(J- R) Q$~\cite{ChouGS24}, and the $\Omega$-stable matrix pair  $(\bar E, \bar A)$ that approximates $(I_n,A)$ using our proposed algorithm.  
The set $\Omega$ is the intersection of a vertical strip, a horizontal strip,  and left and right parabolic regions. \label{fig:ex1}}
\end{center}
\end{figure}
\end{example}

\begin{example}

Let us consider $\Omega$ as the intersection of 
\begin{itemize}
    \item an ellipsoid centered at 
    $(-1,0)$ with horizontal radius 3 and vertical radius of 2, 
 
    \item  a left hyperbolic region centered with semi-major axis $a_h = b_h = 0.5$, and 
    
    \item  a right conic sector centered at $(-3.5,0)$ with angle $\frac{3}{8} \pi$.   
\end{itemize}
As for the previous example, we generate a matrix with $n=10$ and $\epsilon = 1$, as in~\cite{ChouGS24}. 
The algorithm of~\cite{ChouGS24} that imposes $\tilde E = I_n$ provides a solution with relative error $\frac{\|A - \tilde{A} \|_F}{\|A\|_F} = 24.1\%$, with \mbox{$\|A-\tilde A\|_F^2 = 11.70$}. 
Our code approximates $(E,A)$ with $(\bar E, \bar A)$
with a relative error~\eqref{eq:relerr} of $14.2\%$, 
with $\|A-\bar A\|_F^2 + \| E- \bar E\|_F^2 = 4.26$. 
This shows, on a second example, that allowing to modify $E$ can reduce the approximation error significantly. 
Figure~\ref{fig:ex2} illustrates the set $\Omega$ along with the eigenvalues of these decompositions. 
\begin{figure}[ht!] 
\begin{center}
\includegraphics[width=0.6\textwidth]{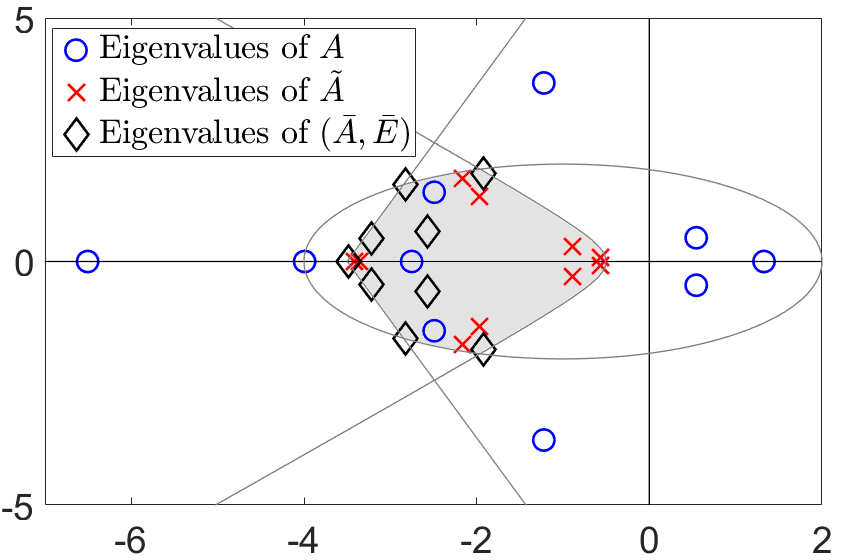}
\caption{Eigenvalues of
$A$, of its $\Omega$-stable approximation $\tilde{A} = 
(J- R) Q$~\cite{ChouGS24}, and the $\Omega$-stable matrix pair  $(\bar E, \bar A)$ that approximates $(I_n,A)$ using our proposed algorithm. 
The set $\Omega$ is the intersection of an ellipsoid, a left hyperbolic region, and a right conic sector.  \label{fig:ex2}}
\end{center}
\end{figure}

\end{example}

\section{Conclusion} 

In this paper, we have proposed a new characterization of $\Omega$-admissible matrix pairs using dissipative Hamiltonian (DH) forms. This allowed us to design a new formulation~\eqref{eq:paramV2} and algorithms to find the nearest $\Omega$-admissible matrix pair. 
The algorithm presented in~\cite{noferini2025finding} has superior average performance compared to our DH-based approach, particularly in systems that exhibit instability and those characterized by extensive Jordan chains.
Also, the algorithm from~\cite{noferini2025finding} is significantly faster on non-Hurwitz systems where our algorithm relies on interior-point methods. The development of faster algorithms, such as first-order methods, for~\eqref{eq:paramV2} remains an area for further investigation. 
However, our algorithm can sometimes provide better or comparable solutions. Furthermore, our algorithm is the first to be able to handle the nearest matrix pair problem for sets $\Omega$ beyond Hurwitz and Schur stability, specifically for LMI regions as detailed  in Table~\ref{tab:specialdhpairlmi} and illustrated in Section~\ref{sec:neelam}.

\small
\bibliographystyle{spmpsci}
\bibliography{ref}

\newpage 

\normalsize 
       
\appendix

\section{Evolution of the relative error for Hurwitz stability} \label{app:figs}

\begin{figure}[ht!]
\begin{center}
\includegraphics[width=\textwidth]{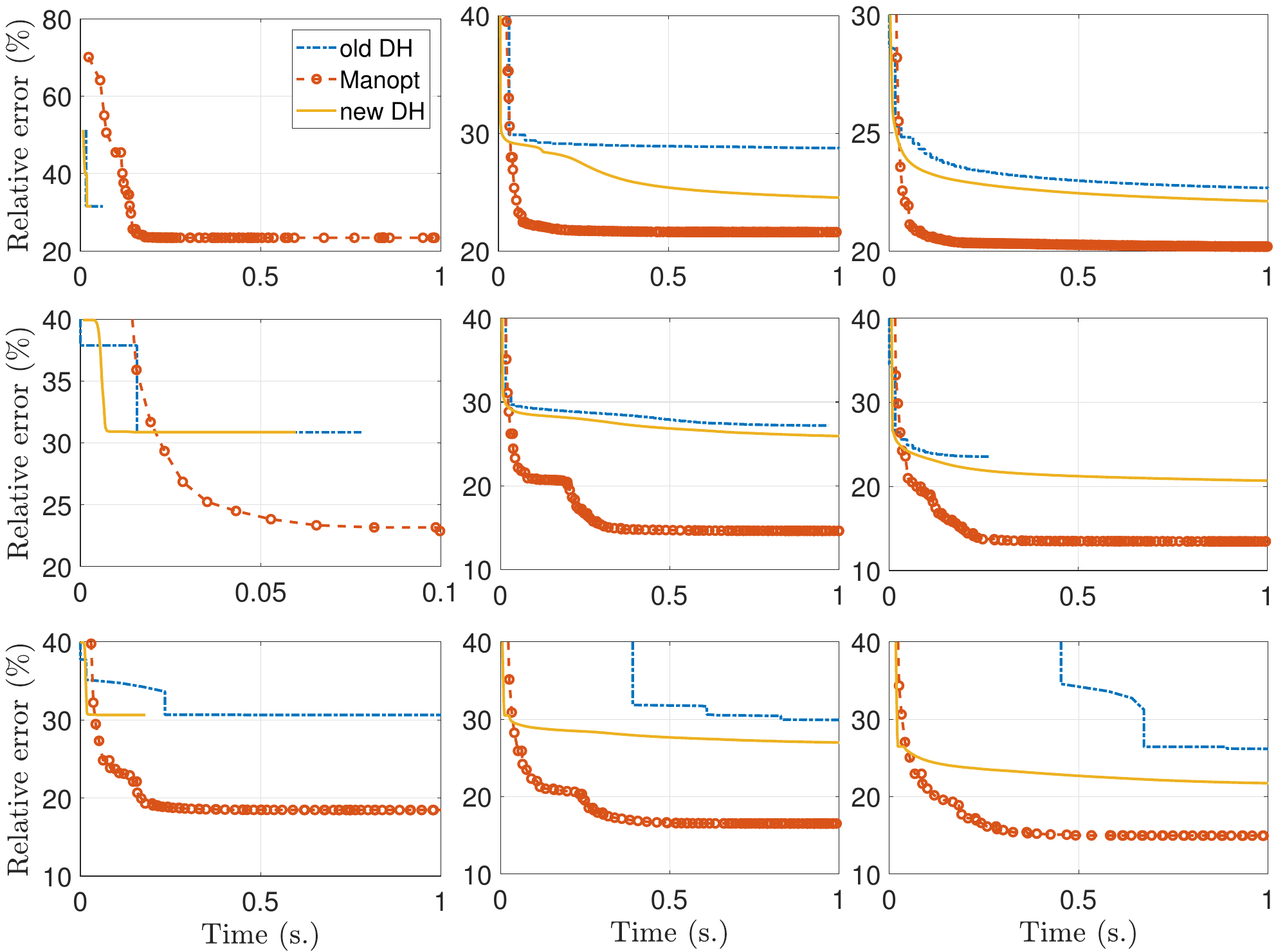}
\caption{Evolution of the relative error of the three algorithms for Grcar matrix pairs in case of Hurwitz stability: from top row to bottom row, $n=10,20,30$; from left to right, $k=1,2,3$. \label{fig:grcarHurwitz}} 
\end{center}
\end{figure}

\newpage 

\begin{figure}[ht!]
\begin{center}
\includegraphics[width=\textwidth]{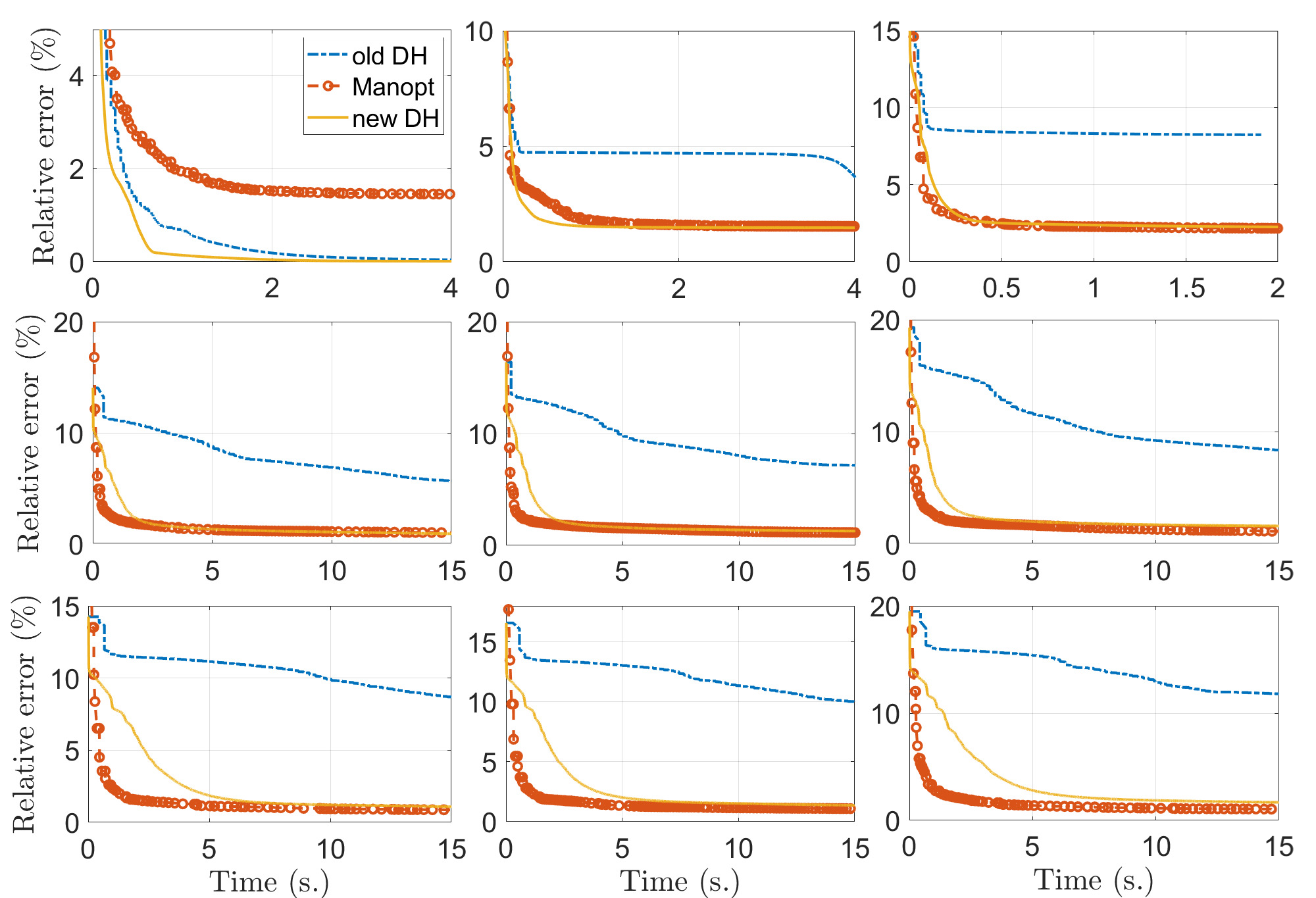}
\caption{Evolution of the relative error of the three algorithms for MSD  systems in case of Hurwitz stability: from top row to bottom row, $n=10,20,30$; from left to right, $\epsilon=0.01,0.05,0.1$. \label{fig:MSDHurwitz}}
\end{center}
\end{figure}

\newpage 

\section{Evolution of the relative error for Schur stability} \label{app:figsSchur}

\begin{figure}[ht!]
\begin{center}
\includegraphics[width=\textwidth]{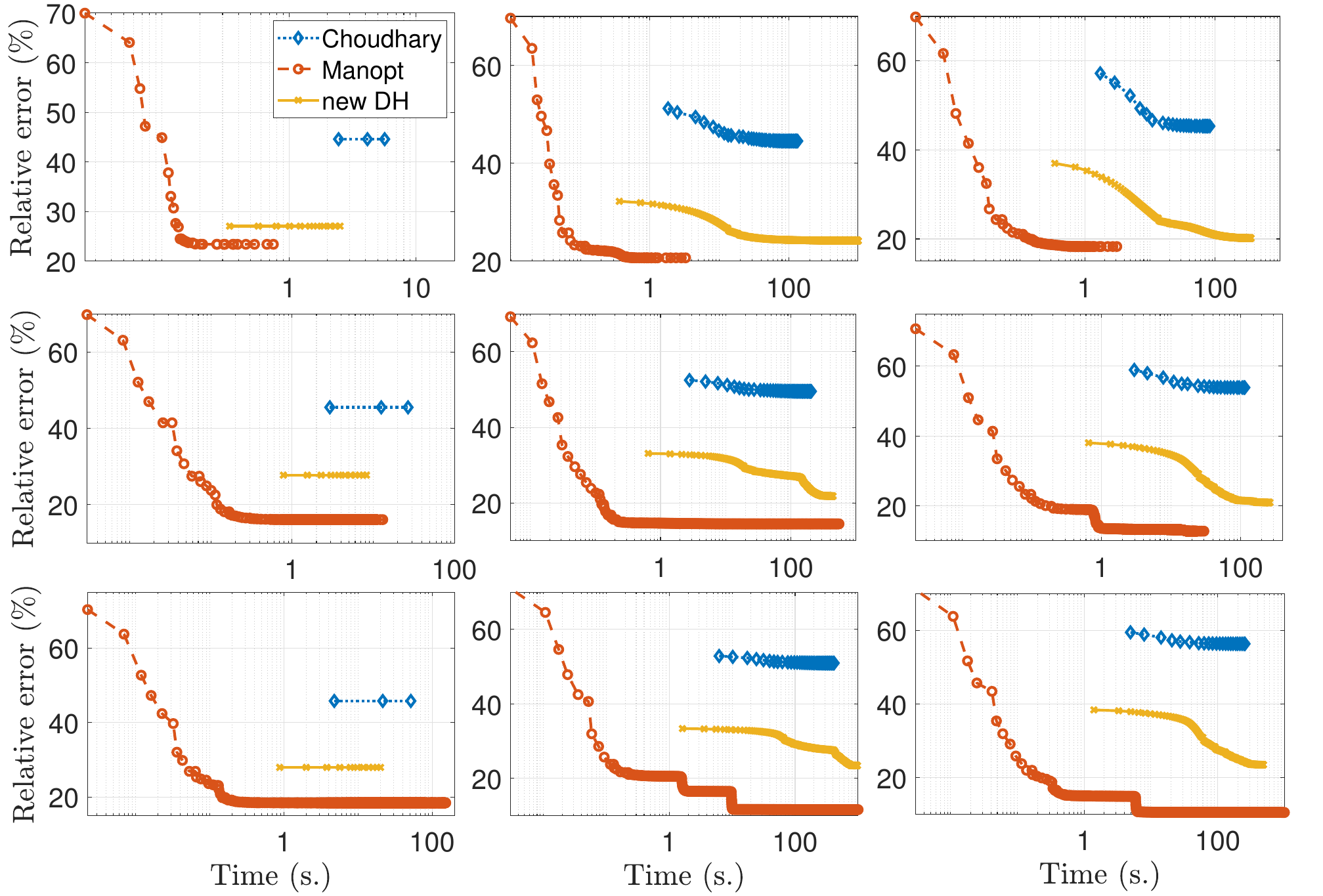}
\caption{Evolution of the relative error of the three algorithms for Grcar matrix pairs in case of Schur stability: from top row to bottom row, $n=10,20,30$; from left to right, $k=1,2,3$. \label{fig:grcarSchur}} 
\end{center}
\end{figure}

\begin{figure}[ht!]
\begin{center}
\includegraphics[width=\textwidth]{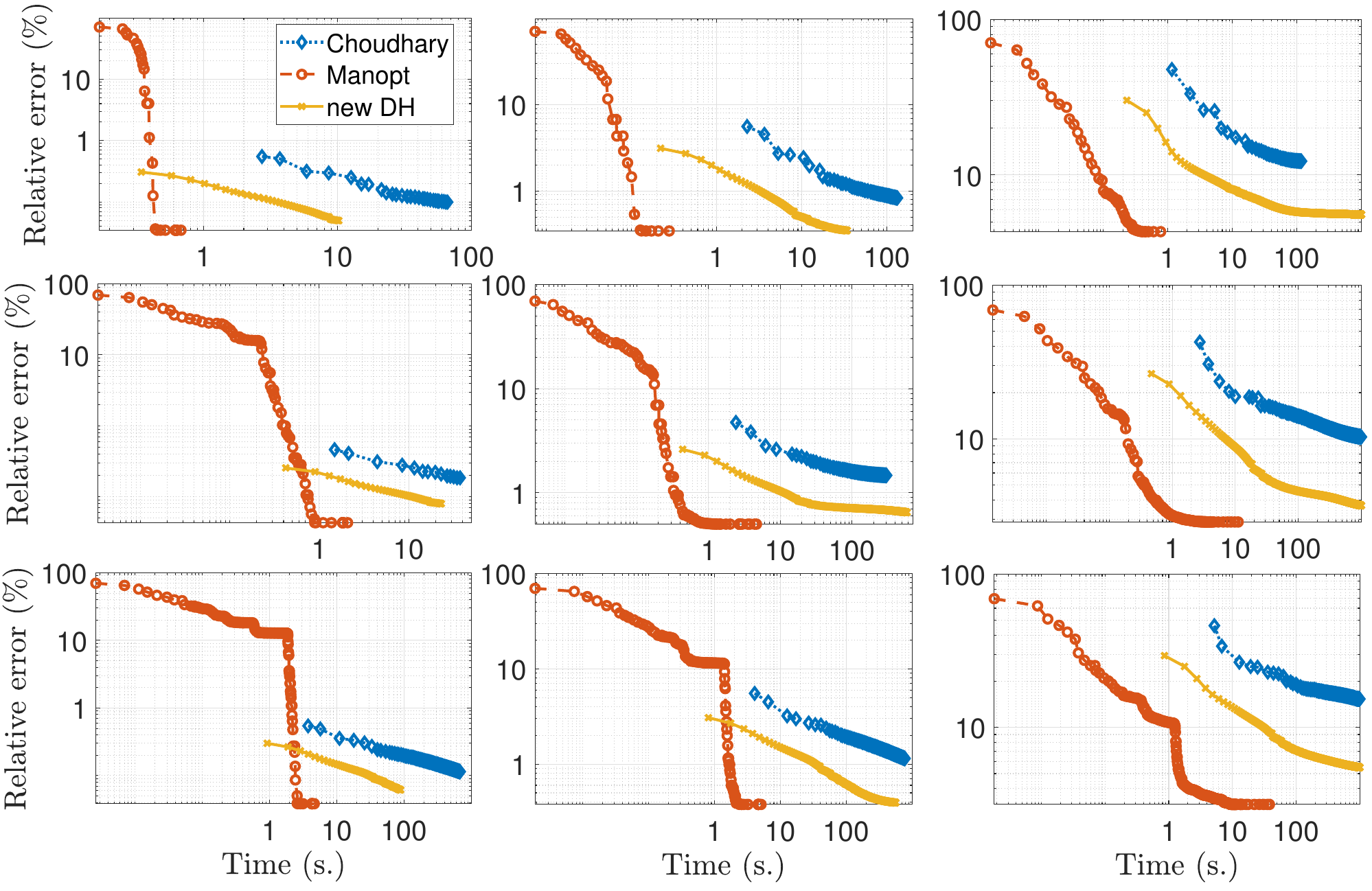}
\caption{Evolution of the relative error of the three algorithms for near-Schur matrix pairs: from top row to bottom row, $n=10,20,30$; from left to right, $\epsilon=0.01,0.1,1$. \label{fig:nearSchur}}
\end{center}
\end{figure}

\end{document}